\documentclass[a4paper,12 pt]{amsart}

\usepackage{amsfonts}
\usepackage{amsmath}
\usepackage{amsmath,amsfonts,amssymb,amsthm,latexsym}
\usepackage{graphicx}
\usepackage{hyperref}

\newcommand{\RR}{\mathbb{R}}
\newcommand{\CC}{\mathbb{C}}
\newcommand{\NN}{\mathbb{N}}
\newcommand{\ZZ}{\mathbb{Z}}
\newcommand{\OO}{\mathcal{O}}

\newtheorem{Tw}{Theorem}
\newtheorem{Le}{Lemma}
\newtheorem{Stw}{Proposition}

\newtheorem{Wn}{Corollary}

\theoremstyle{remark}
\newtheorem{Uw}{Remark}

\theoremstyle{definition}
\newtheorem{Df}{Definition}

\begin{document} 
\title[A Dirichlet type problem]{A Dirichlet type problem for complex polyharmonic functions}
\author{Hubert Grzebu{\l}a}
\address{Faculty of Mathematics and Natural Sciences,
College of Science\\
Cardinal Stefan Wyszy\'nski University\\
W\'oycickiego 1/3,
01-938 Warszawa, Poland}
\email{h.grzebula@student.uksw.edu.pl}
\author{S{\l}awomir Michalik}
\address{Faculty of Mathematics and Natural Sciences,
College of Science\\
Cardinal Stefan Wyszy\'nski University\\
W\'oycickiego 1/3,
01-938 Warszawa, Poland}
\email{s.michalik@uksw.edu.pl}
\urladdr{\url{http://www.impan.pl/~slawek}}

\begin{abstract}
We extend holomorphically polyharmonic functions on a real ball to  a complex set being the union of rotated balls.
We solve a Dirichlet type problem for complex polyharmonic functions with the boundary condition given on the union of rotated spheres.
 \end{abstract}
\maketitle  
\section{Introduction}
Boundary value problems for polyharmonic functions have recently been extensively studied (see \cite{G-G-S} and the references given there).
In these problems the boundary conditions are expressed by linear differential operators (see \cite[Section 2.3]{G-G-S}) such as
the normal derivatives in the case of Dirichlet boundary conditions or the iterated Laplacians in the case of Navier boundary conditions.

Nevertheless, we study a quite different Dirichlet type problem, where the boundary conditions are given only in the terms of the solution.
Here the conditions on the normal derivatives or on the iterated Laplacians of the solution are replaced by the conditions given in terms of
values of the solution on rotated spheres.  

More precisely, our problem is to find a polyharmonic function $u$ of order $ p $ on the union of the rotated unit balls  
$\bigcup_{k=0}^{p-1}e^\frac{k\pi i}{p}B$, $B\subseteq \RR^n$, such that $ u $ is continuous in $\bigcup_{k=0}^{p-1}e^\frac{k\pi i}{p}\overline{B} $ 
and satisfies the boundary conditions
$$
u(x)=f_k(x)\quad \textrm{for}\quad  x\in e^\frac{k\pi i}{p}\partial B,
$$  
where $ p\in \NN $ and the functions $f_k$ are given and continuous in $e^\frac{k\pi i}{p}\partial B$, for $k=0,1,\dots,p-1$.
Throughout this paper, all such problems will be concisely written as
\begin{equation}
\label{eq:1}
   \begin{cases}
   \Delta^pu(x)=0, & \text{} x\in \bigcup_{k=0}^{p-1}e^\frac{k\pi i}{p}B \\
   u(x)=f_k(x), & \mbox{} x\in e^\frac{k\pi i}{p}\partial B,\ k=0,1,...,p-1.
   \end{cases}
   \end{equation}
   
We see that the condition of polyharmonicity is put on the union of the rotated unit balls.
But in view of Lemma \ref{le:1}, every polyharmonic function on the unit ball $B$ can be polyharmonically extended to the union of the rotated unit
balls. Hence, it is sufficient to assume in (\ref{eq:1}) that $u$ is a polyharmonic function on $B$ and its holomorphic extension is
continuous in $\bigcup_{k=0}^{p-1}e^\frac{k\pi i}{p}\overline{B}$.

The motivation for the study of this problem comes from the Pizzetti-type formula for the operator $\Delta^p$, which is a generalisation
of the classical Pizzetti formula (see \cite{L0,L1}), given
by the second author \cite[Corollary 3]{S-M}. By this formula,
the integral mean of $u$ over the rotated spheres $x+\bigcup_{k=0}^{p-1}e^\frac{k\pi i}{p}\partial B(0,r)$ given by
$$
M_{\Delta^p}(u;x,r):=\frac{1}{p\omega_n}\sum_{k=0}^{p-1}\int\limits_{\partial B}u(x+e^{\frac{k\pi i}{p}}r\zeta)\,dS(\zeta)
$$
has the expansion
$$
 M_{\Delta^p}(u;x,r)
 =\sum_{j=0}^{\infty}\frac{\Delta^{pj}u(x)}{4^{pj}(n/2)_{pj}(pj)!}r^{2pj},
 $$
 where $(a)_k:=a(a+1)\cdots(a+k-1)$, for $k\in\NN$, is the Pochhammer symbol and $\omega_n $ denotes the area of the unit sphere in $ {\RR}^n $.
See also \cite[Theorem 1]{L}.

Hence, the mean value property $u(x)=M_{\Delta^p}(u;x,r)$ holds (see \cite[Remark 6]{S-M}) for every
 polyharmonic function $u$ of order $p$ on the closed rotated balls $x+\bigcup_{k=0}^{p-1}e^\frac{k\pi i}{p}\overline{B}(0,r)$.

In particular, it means that the value $u(0)$ of the polyharmonic function $u$ is uniquely determined by the boundary values on the rotated unit
spheres $\bigcup_{k=0}^{p-1}e^\frac{k\pi i}{p}\partial B$. Here the question arises how to determine the values of $u$ on
the whole union of the rotated balls $\bigcup_{k=0}^{p-1}e^\frac{k\pi i}{p}B$ in terms of the boundary values
or, in other words, how to solve (\ref{eq:1}).

The answer is given in the main theorem (Theorem \ref{th:1}), which says that the unique solution of (\ref{eq:1}) is given by the sum
of Poisson type integrals
\begin{equation*}
u(x) =\frac{1}{p\omega_n}\sum_{k=0}^{p-1} 
\int\limits_{\partial B} \frac{1-|x|^{2p}}{|e^\frac{-k\pi i}{p}x-\zeta|^n }f_k(e^\frac{k\pi i}{p}\zeta) \,dS(\zeta).
\end{equation*}

The paper is organised as follows. In the next section we write several remarks about the notation used throughout the paper.

In Section 3 we examine the holomorphic extension of polyharmonic functions. We prove that every polyharmonic function on the unit ball
$ B $ can be extended to  the set $ e^{i\varphi}B$, for each $\varphi \in \RR $ (Lemma \ref{le:2}).
Here, we use the Almansi expansion (Proposition \ref{pr:1}) and the Poisson integral. We recall also the definitions of the Lie ball
and the Lie sphere
and Siciak's theorem (Proposition \ref{pr:2})  about the holomorphic extension of harmonic functions on the ball $B$ to the Lie ball $LB$.

In Section 4 we prove the main theorem (Theorem \ref{th:1}) about the existence and uniqueness of the solution of
the main problem (\ref{eq:1}).
First we find a form of polyharmonic functions on the ball, which is suitable for solving this problem (Lemma \ref{le:2}).

In Section 5 we note two corollaries of the main theorem. Firstly, we solve the Dirichlet problem in the general case,
i.e. the problem for the ball $ B(a,r)$ with boundary conditions put on the set $ a+\bigcup_{k=0}^{p-1} e^\frac{k\pi i}{p}\partial B(0,r) $
(Corollary \ref{th:2}). Next, we solve the exterior problem corresponding to the interior problem (\ref{eq:1}) (Corollary \ref{th:3}).
Here polyharmonicity is required on the set $\RR^n\setminus B$ with the same boundary conditions as in the problem (\ref{eq:1}).

In the final section we indicate some further applications of the main result, which will be studied in the subsequent paper

\section{Preliminaries}
Before solving the problem (\ref{eq:1}) we recall some basic notations.

We define the real norm
\begin{equation*}
|x|= \left( \sum_{j=1}^nx_j^2 \right)^{1/2}\quad\textrm{for}\quad x=(x_1,\dots,x_n)\in \RR^n
\end{equation*}
and the complex norm
\begin{equation*}||z||= \left(\sum_{j=1}^n|z_j|_{\CC}^2 \right)^{1/2} \quad\textrm{for}\quad z=(z_1,\dots,z_n)\in \CC^n
\end{equation*} 
with $ |z_j|_{\CC}^2=z_j \overline{z}_j$. We will also use the complex extension of the real norm for complex vectors: 
\begin{gather*}
|z|=\left( \sum_{j=1}^nz_j^2 \right)^{1/2}\quad\textrm{for}\quad z=(z_1,\dots,z_n)\in\CC^n.
\end{gather*}
By a square root in the above formula we mean the principal square root, where a branch cut is taken along the nonpositive real axis.
Obviously the function $|\cdot|$ is not a norm in $\CC^n$, because it is complex valued and hence the function $|z-w|$ is not a metric on $\CC^n$.

We will  consider mainly complex vectors of the form $ z=e^{i\varphi}x $, that is real vectors $x\in \RR^n $ rotated by the angle
$ \varphi \in (-\pi,\pi]$. In this case we have 
$|e^{i\varphi}x|=|e^{i\varphi}|_{\RR}|x|$,
where $|e^{i\varphi}|_{\RR}$ denotes the principal value of the square root $\sqrt{e^{2i\varphi}}$, i.e.
\begin{equation*}
|e^{i\varphi}|_{\RR}=\sqrt{e^{2i\varphi}}=\left\{
\begin{array}{lll}
 e^{i(\varphi+\pi)}=-e^{i\varphi} & \textrm{for} & \varphi \in (-\pi,-\frac{\pi}{2}]\\
 e^{i\varphi} & \textrm{for} & \varphi \in (-\frac{\pi}{2},\frac{\pi}{2}]\\
 e^{i(\varphi-\pi)}=-e^{i\varphi} & \textrm{for} & \varphi \in (\frac{\pi}{2},\pi].
\end{array}
\right. 
\end{equation*}
For the set $G\subseteq\RR^n$ and the angle $\varphi\in\RR$ we will consider the rotated set defined by
$$e^{i\varphi}G:=\{e^{i\varphi}x:x\in G\}.$$
Let $G$  be an open set in $ \RR^n $. We denote by $\mathcal{A}(G)$ the space of analytic functions on $ G $.
Similarly we say that $f\in \mathcal{A}(e^{i\varphi}G)$ if and only if $f_{\varphi}(x):=f(e^{i\varphi}x)\in \mathcal{A}(G)$. We call
$$\mathcal{A}_{\Delta}(G):=\{f\in \mathcal{A}(G):\Delta_x f=0 \}$$ 
the space of harmonic functions on $ G $, where $\Delta_x $ denotes the Laplacian in $\RR^n$.
Analogously we define the family $\mathcal{A}_{\Delta}(e^{i\varphi}G)$ of harmonic functions on $e^{i\varphi}G$. Observe that
$f\in \mathcal{A}_{\Delta}(e^{i\varphi}G)$ if and only if $f_{\varphi}\in\mathcal{A}_{\Delta}(G)$. 
Similarly, replacing the Laplace operator $\Delta_x$ by its $p$-th iteration $\Delta_x^p$ in the above definitions,
we introduce the spaces of polyharmonic functions of degree $p$, that is $ \mathcal{A}_{\Delta^p}(G)$
and $\mathcal{A}_{\Delta^p}(e^{i\varphi}G)$.

Now let $U$ be an open set in $\CC^n$. We denote by $\mathcal{O}(U)$ the space of holomorphic functions on $U$. We call 
$$\mathcal{O}_{\Delta}(U)=\{f\in \OO(U)\colon \Delta_z f=0\}$$ 
the space of complex harmonic functions on $U$, where $\Delta_z$ denotes the complex Laplace operator in $\CC^n$
$$\Delta_z=\frac{\partial^2}{\partial z_1^2}+\frac{\partial^2}{\partial z_2^2}+\cdots+\frac{\partial^2}{\partial z_n^2}.$$
 
Similarly we introduce the space $\OO_{\Delta^p}(U)$ of polyharmonic complex functions on $U$.

Let us note also that $f\in \mathcal{A}(G)$ (resp. $f\in\mathcal{A}_{\Delta}(G)$, $ f\in\mathcal{A}_{\Delta^p}(G)$) if and only if there is
a complex open neighbourhood $U$ of $G$ (i.e. $U$ is open in $\CC^n$ and $G\subseteq U$) such that $f$ extends to the holomorphic function
$\widetilde{f}\in \mathcal{O}(U)$ (resp. $\widetilde{f}\in\mathcal{O}_{\Delta}(U)$, $\widetilde{f}\in\mathcal{U}_{\Delta^p}(U)$).
The same properties hold if we replace the real set $G$ by the rotated one $e^{i\varphi}G$. Such function $\widetilde{f}$ will be called
the \emph{complex holomorphic} (resp. \emph{harmonic}, \emph{polyharmonic}) \emph{extension of} $f$.

Throughout the paper $ B(a,r) $ denotes the real open ball with the centre at  $a\in {\RR}^n$ and radius $ r>0 $.
We denote the unit ball $ B(0,1) $ simply by $B$.

\section{Analytic extension}
In this section we will discuss the analytic extension of polyharmonic functions.
First we recall the finite Almansi expansion which will be extensively used in our study of polyharmonic functions.

\begin{Stw}[Finite Almansi expansion, {\cite[Proposition 1.2, Proposition 1.3]{A-C-L}}] 
\label{pr:1}
If $u\in \mathcal{A}_{\Delta^p}(B)$, then there exist unique functions  $h_0,h_1,...,h_{p-1}\in \mathcal{A}_{\Delta}(B)$ such that
\begin{gather}
\label{eq:2}
u(x)=h_0(x)+|x|^2h_1(x)+...+|x|^{2(p-1)}h_{p-1}(x) \quad \textrm{for} \quad x\in B.
\end{gather}
Moreover, any function $u$ of the form (\ref{eq:2}) is a polyharmonic function of order $p$ on $B$.
\end{Stw}
 
Now we can prove the lemma about the holomorphic extension. 

\begin{Le}
\label{le:1}
Let $\varphi\in\RR$ and $u\in \mathcal{A}_{\Delta^p}(B)$. Then 
the function $u$ has a holomorphic extension to the set $\{z\in\CC^n\colon\ z=e^{i\psi}x,\ \psi\in\RR,\ x\in B\}$,
whose restriction $u_{\varphi} $ to $e^{i\varphi}B$ is polyharmonic of order $p$, i.e. $u_{\varphi}\in \mathcal{A}_{\Delta^p}(e^{i\varphi}B)$.
\end{Le}
\begin{proof}
First, we assume that $u\in \mathcal{A}_{\Delta}(B)$. Fix $r \in (0,1)$ and let $u_r(x):=u(rx)$ for $x\in B$.
Observe that $u_r$ is harmonic on $B$ and
continuous on its closure $\overline{B}$. Then, by the Poisson formula (\cite[Theorem 1.17]{A-B-R}), for every $x\in B$ we have
\begin{equation}
\label{eq:Poisson}
u_r(x)=\frac{1}{\omega_n}\int\limits _{\partial B} \frac{1-|x|^2}{\vert x-\zeta \vert^n}u(r\zeta)\,dS(\zeta),
\end{equation}
where $\omega_n=\frac{2\pi^{\frac{n}{2}}}{\Gamma(\frac{n}{2})}$ denotes the area of the unit sphere $\partial B$ in $\RR^n$ and
$dS(\zeta)$ is the surface measure on $\partial B$.

Let us denote by $u_{r,\varphi}(x)$ the function defined by the right hand side of (\ref{eq:Poisson}) in the case when $x\in e^{i\varphi}B$.
We will show that it is a well-defined analytic continuation of $u_r$ on the set $e^{i\varphi}B$ and
$u_{r,\varphi}\in\mathcal{A}_{\Delta}(e^{i\varphi}B)$. To this end let us
consider the auxiliary function $v_r$ on $B$ given by
 $$v_r(x):=u_{r,\varphi}(e^{i\varphi} x)=\frac{1}{\omega_n}\int\limits _{\partial B}
\frac{1-e^{2i\varphi}|x|^2}{\vert e^{i\varphi}x-\zeta \vert^n}u(r\zeta)\,dS(\zeta).$$
We will show that the function $v_r$ is well defined and real-analytic in $B$. For this it is sufficient to prove that there exist
$\varepsilon>0$ such that
$$
\inf_{\zeta\in\partial B}{\rm dist}\left( |e^{i\varphi}x-\zeta|^n,0 \right) \geq \varepsilon.
$$
By the compactness of $\partial B$ it is equivalent to the condition
\begin{equation}
\label{eq:compactness}
\vert e^{i\varphi}x-\zeta \vert^n\neq 0 \quad \textrm{for} \quad \zeta \in \partial B.
\end{equation}
Observe that, since $|\cdot|$ is not a norm in $\CC^n$, it is not sufficient to show that $e^{i\varphi}x\neq\zeta$.
So suppose, contrary to (\ref{eq:compactness}), that
$\vert e^{i\varphi}x-\zeta \vert^n = 0$ for some $ \zeta \in \partial B $. 
Then
$$ \sum_{k=1}^n \left( e^{i\varphi}x_k-\zeta_k \right) ^2=0$$
or equivalently
$$e^{i\varphi}|x|^2+e^{-i\varphi}-2x\zeta=0.$$
Since the imaginary part of the left hand side of the above equation has to be equal to $0$, we have 
$\vert x\vert =1$ for $\varphi\neq k\pi $, which is a contradiction. If $\varphi=k\pi $, then for $k=0,2,4,\dots$ we get  
$\vert e^{i\varphi}x-\zeta \vert =\vert x-\zeta \vert  \neq 0$ and for $k=1,3,5,\dots$ we get 
$\vert e^{i\varphi}x-\zeta \vert=\vert x+\zeta \vert \neq 0$, which is a contradiction too, as desired.  

From the above considerations we conclude that $u_{r,\varphi}\in\mathcal{A}(e^{i\varphi}B)$.
Moreover, if $\widetilde{u}_r$ and $\widetilde{u}_{r,\varphi}$ denote the complex holomorphic extensions of $u_r$ and $u_{r,\varphi}$ respectively,
then by (\ref{eq:Poisson})
there exist an open complex neighbourhood $U$ of the origin, such that $\widetilde{u}_r(z)=\widetilde{u}_{r,\varphi}(z)$ for $z\in U$.
It means that $u_{r}$ is analytically extended to $u_{r,\varphi}\in\mathcal{A}(e^{i\varphi}B)$ and also $u$ is analytically extended to
$u_{\varphi}\in\mathcal{A}(e^{i\varphi}B(0,r))$.  Hence also $\Delta u$ is analytically extended to
$\Delta u_{\varphi}\in\mathcal{A}(e^{i\varphi}B(0,r))$. On the other hand, since $\Delta u=0$ on $B$, by the uniqueness
of the analytic
continuation also $\Delta u_{\varphi}=0$ on $e^{i\varphi}B(0,r)$. So, we conclude that $ u_{\varphi}\in \mathcal{A}_{\Delta}(e^{i\varphi}B(0,r))$.
Since it holds for arbitrary $r\in(0,1)$, $u_{\varphi}\in \mathcal{A}_{\Delta}(e^{i\varphi}B)$ is the analytic extension of $u$.

If $u$ is polyharmonic on $B$, then by Proposition \ref{pr:1} and by the above considerations we conclude that $u_{\varphi}$ is also polyharmonic on
$e^{i\varphi}B$.
\end{proof}
 
Now we show that the analytic extension to the set $ e^{i\varphi}B $ follows also from Siciak's theorem. First we introduce the notion
of the Lie ball and sphere.
\begin{Df}
The sets
$$ LB(0,r):=\{z\in \mathbb{C}^n\colon L(z)<r \},$$
$$ LS(0,r):=\{e^{i\varphi }x\colon \varphi \in \mathbb{R},\ x\in \partial B(0,r) \},$$
where  
$$ L(z)=\sqrt{||z||^2+\sqrt{||z||^4-|z^2|_{\CC}^2}}, $$
with $z^2=z\cdot z$, are called the \emph{Lie ball} and the \emph{Lie sphere} of radius $ r $, respectively.

Throughout the paper we will denote the Lie ball $ LB(0,1)$ and the Lie sphere $LS(0,1)$ by $LB$ and $LS$, respectively.
\end{Df}

\begin{Uw}
 The Lie ball $LB(0,r)$ is E. Cartan's classical domain of type $4$ studied by Hua \cite{H}. In particular he proved in \cite{H} that
 the Lie sphere $LS(0,r)$
 is the Shilov boundary of the closure of the Lie ball $\overline{LB}(0,r)=\{z\in \mathbb{C}^n\colon L(z)\leq r \}$ (i.e. $LS(0,r)$ is the
 smallest closed subset of the boundary of $\overline{LB}(0,r)$, where every holomorphic function on $LB(0,1)$ and continuous on
 $\overline{LB}(0,r)$ reaches its maximum).
 \end{Uw}
 \begin{Uw}
 Since $e^{i\varphi }x=e^{i(\varphi+\pi)}(-x)$, we get 
 $$LS(0,r)\cong (\RR/2\pi\ZZ\times\partial B(0,r))/\sim,$$
 where $\sim$ is the equivalence relation defined by
 $$
 (\varphi, x)\sim (\varphi+\pi, -x)
 $$ 
 (see Hua \cite{H} or Morimoto \cite{M2}).
\end{Uw}

The following proposition holds.
\begin{Stw} [Siciak's theorem \cite{S}, see also {\cite[Theorem 1]{F-M}} and {\cite[Theorem 3.38]{M}}]
\label{pr:2}
The Lie ball  $ LB(0,r) $ is the harmonic hull of the real ball $B(0,r)$, that is, every harmonic function on $ B(0,r) $ is holomorphically extended to the Lie ball $ LB(0,r) $.
\end{Stw}
 
Let us observe that the set $ e^{i\varphi}B $ is contained in the Lie ball $ LB(0,1) $. Indeed, if $ x\in B $ then
 $$ L(e^{i\varphi}x)=\sqrt{|x|^2+\sqrt{|x|^4-|e^{i\varphi}xe^{i\varphi}x|_{\CC}^2}}=|x|<1. $$ 
From the above and from Proposition \ref{pr:2} we conclude that each harmonic function on  $ B $ is holomorphically extended to $ e^{i\varphi}B$.

\section{The main result}
Now we will prove an auxiliary lemma, which we will use to solve our Dirichlet problem.
\begin{Le}
\label{le:2}
If $\Delta^p u=0$ on $B$, then there exist unique functions $g_0,g_1,\dots,g_{p-1}$, each harmonic on $B$ such that  
\begin{equation} 
\label{eq:3}
u(x)=\sum_{k=0}^{p-1}\frac{1-|x|^{2p}}{1-e^\frac{2k\pi i}{p}|x|^2}g_k(x)\quad \textrm{for}\quad x\in B.
\end{equation}
\end{Le}
\begin{proof}
Let $ u $ be a polyharmonic function of order $p$ on $ B $.
We will find uniquely determined harmonic functions $g_0,g_1,\dots,g_{p-1}$ on $B$, which satisfy (\ref{eq:3}).

Denote by $a_k=a_k(|x|^2)$ the coefficient in front of the function $g_k$ in (\ref{eq:3}) for $k=0,1,\dots,p-1$. It is easy to observe that
\begin{gather*}
a_k = 1+e^\frac{2k\pi i}{p}|x|^2+\dots+e^\frac{2k(p-1)\pi i}{p}|x|^{2(p-1)}\quad \textrm{for}\quad k=0,1,\dots,p-1.
\end{gather*}
If we put the above formula to (\ref{eq:3}), we obtain 
\begin{eqnarray*}
 u(x) & = & g_0(x)+g_1(x)+\dots+g_{p-1}(x) \\
 & + & |x|^2 \left[ g_0(x)+e^\frac{2\pi i}{p}g_1(x)+\dots+e^\frac{2(p-1)\pi i}{p}g_{p-1}(x)\right] \\
 & \dots & \\
 & + & |x|^{2(p-1)} \left[ g_0(x)+e^\frac{2(p-1)\pi i}{p}g_1(x)+\dots+e^\frac{2(p-1)^2\pi i}{p}g_{p-1}(x)\right] 
\end{eqnarray*}
for $x\in B$.

By Proposition \ref{pr:1} we know that there exist unique harmonic functions $h_k$ on $B$, $k=0,1,\dots,p-1$, such that formula (\ref{eq:2}) is valid.
Substituting the above equality to (\ref{eq:2}) and comparing corresponding coefficients standing by the same powers of $|x|^2$,
we get the following system of linear equations:
\begin{displaymath}
\left\{ \begin{array}{ll}
g_0(x)+g_1(x)+\dots+g_{p-1}(x) =h_0(x)   \\
g_0(x)+e^\frac{2\pi i}{p}g_1(x)+\dots+e^\frac{2(p-1)\pi i}{p}g_{p-1}(x) = h_1(x) \\
\dots\\
g_0(x)+e^\frac{2(p-1)\pi i}{p}g_1(x)+\dots+e^\frac{2(p-1)^2\pi i}{p}g_{p-1}(x) = h_{p-1}(x).
\end{array} \right.
\end{displaymath}

The fundamental matrix of the above system has the form 
$$A=\left[ e^{\frac{2kl\pi i}{p}}\right] _{k,l=0}^{p-1}.$$
Since $A$ is a Vandermonde matrix, we get (see \cite[p.~9]{M-T})
\begin{equation*}
\det A = \prod_{0\leq k< l\leq p-1}\left(e^\frac{2l\pi i}{p}-e^\frac{2k\pi i}{p} \right)         
\end{equation*}
Observe that $e^\frac{k\pi i}{p}\neq e^\frac{l\pi i}{p}$ for each integers $k,l$ such that $k\neq l+mp$ for any $m\in \mathbb{N}$, hence
the matrix $A$ is nonsingular. Therefore
\begin{equation}
\label{eq:5}
G=A^{-1}H,
\end{equation}
where
$$
   G=\left[\begin{array}{ccccc}
         g_0(x) \\
         \vdots \\
         g_{p-1}(x)  
         \end{array}
   \right],\quad    H=\left[\begin{array}{ccccc}
         h_0(x) \\
         \vdots \\
         h_{p-1}(x)  
         \end{array}
   \right].
   $$

By (\ref{eq:5}) there exist functions  $g_k$, which are harmonic in $B$ and are uniquely determined by the functions $h_k$.  
\end{proof}
\begin{Uw}
\label{re:1}
By Lemma \ref{le:1}, it follows that the Dirichlet problem given by (\ref{eq:1}) can be written by
\begin{equation}
\label{eq:6}
   \begin{cases}
   \Delta^pu(x)=0, & \text{} x\in B \\
   u(x)=f_k(x), & \mbox{} x\in e^\frac{k\pi i}{p}\partial B,\ k=0,1,\dots,p-1.
   \end{cases}
   \end{equation}
\end{Uw}

Now we are ready to state the main result of the paper
\begin{Tw}[The main theorem]
\label{th:1}
The Dirichlet problem (\ref{eq:6}) has the unique solution given by
\begin{equation}
\label{eq:7}
u\left( x \right) =\frac{1}{p\omega_n}\sum_{k=0}^{p-1} 
\int\limits_{\partial B} \frac{1-|x|^{2p}}{|e^\frac{-k\pi i}{p}x-\zeta|^n }f_k(e^\frac{k\pi i}{p}\zeta) \,dS(\zeta).
\end{equation}
\end{Tw}
\begin{proof}
By Lemma \ref{le:2} we will look for the solution of the problem  (\ref{eq:6}) in the form (\ref{eq:3}).
It is easy to observe that for each $ k,l=1,\dots,p $ and for $x\in \partial B$ we have
\begin{equation}
\label{eq:8}
  a_l(|e^\frac{(p-k)\pi i}{p}x|^2)=\sum_{j=0}^{p-1}e^\frac{2jl\pi i}{p}e^\frac{2j(p-k)\pi i}{p}|x|^{2j} =\begin{cases}
                                                                                                              p\ \textrm{for}\ k=l\\
                                                                                                              0\ \textrm{for}\ k\neq l
                                                                                                             \end{cases}.
\end{equation}
By Lemma \ref{le:1} we extend the function given by
 (\ref{eq:3}) to the set $ \bigcup_{k=0}^{p-1}e^\frac{k\pi i}{p}B $. Putting the boundary conditions from (\ref{eq:6}) and
 applying formula (\ref{eq:8}) we obtain
\begin{equation*}
g_k \left( e^\frac{(p-k)\pi i}{p} x \right) = \frac{1}{p} u \left( e^\frac{(p-k)\pi i}{p} x \right) 
= \frac{1}{p} f_{p-k}\left( e^\frac{(p-k)\pi i}{p} x \right)
\end{equation*}
for $x\in \partial B$ and $k=1,2,\dots,p$ with the notation $g_p:=g_0$.

Let us denote  
\begin{equation}
\label{eq:10}
\widehat{g}_k(x):=g_k(e^\frac{(p-k)\pi i}{p} x)\qquad \textrm{for}\quad x\in B, \quad k=1,2,\dots,p
\end{equation} 
and
\begin{equation}
\label{eq:11}
\widehat{f}_k(x):=f_k(e^\frac{k\pi i}{p} x)\qquad \textrm{for}\quad x\in B, \quad k=0,1,\dots,p-1.
\end{equation}
We see that the functions $\widehat{g}_k$ are harmonic on $ B$,
and the functions  $\widehat{f}_k$ are continuous on $\partial B$. In this way we obtain  $p$ Dirichlet problems for harmonic
functions $\widehat{g}_k$: 
$$
   \begin{cases}
   \Delta \widehat{g}_k(x)=0, & \text{} x\in B \\
   \widehat{g}_k(x)=\frac{1}{p}\widehat{f}_{p-k}(x), & \mbox{} x\in \partial B
   \end{cases}
   $$
for $k=1,2,\dots,p$.
 
The solutions of the above problems $\widehat{g}_k$ have the form of the Poisson integrals:
$$\widehat{g}_k(x)=\frac{1}{p\omega_n}\int\limits_{\partial B}\frac{1-|x|^2}{|x-\zeta|^n }\widehat{f}_{p-k}(\zeta)\,dS(\zeta)$$
for $x\in  B$ and $k=1,2,\dots,p$.
  
From (\ref{eq:10}) and (\ref{eq:11}) we have
$$
g_k\left( e^\frac{(p-k)\pi i}{p}x \right) =
\frac{1}{p\omega_n}\int\limits_{\partial B}\frac{1-|x|^2}{|x-\zeta|^n}f_{p-k} \left( e^\frac{(p-k)\pi i}{p}\zeta \right) dS(\zeta)
$$
for $x\in B$, $k=1,2,\dots,p$. Therefore
\begin{equation*}
g_k\left( x \right) =\frac{1}{p\omega_n}\int\limits_{\partial B}
\frac{1-e^\frac{2(k-p)\pi i}{p}|x|^2}{|e^\frac{(k-p)\pi i}{p}x-\zeta|^n }f_{p-k} \left( e^\frac{(p-k)\pi i}{p}\zeta \right) dS(\zeta)
\end{equation*}
for $x\in e^\frac{(p-k)\pi i}{p}   B$ and $k=1,2,\dots,p$. We extend the obtained functions  $g_k$ to the set
$\bigcup_{k=0}^{p-1}e^\frac{k\pi i}{p}B$ and then we put them to (\ref{eq:3}): 
\begin{eqnarray*}
u(x) & = & \frac{1}{p\omega_n}\sum_{k=1}^{p}\frac{1-|x|^{2p}}{1-e^\frac{2k\pi i}{p}|x|^2}\int\limits_{\partial B}
\frac{1-e^\frac{2(k-p)\pi i}{p}|x|^2}{|e^\frac{(k-p)\pi i}{p}x-\zeta|^n }f_{p-k} \left( e^\frac{(p-k)\pi i}{p}\zeta \right) dS(\zeta)\\
&=& \frac{1}{p\omega_n}\sum_{k=1}^{p}\int\limits_{\partial B}\frac{1-|x|^{2p}}{|e^\frac{(k-p)\pi i}{p}x-\zeta|^n }f_{p-k}
\left( e^\frac{(p-k)\pi i}{p}\zeta \right)dS(\zeta).
\end{eqnarray*}
Substituting $ m=p-k $ we get the desired formula (\ref{eq:7}). 
The uniqueness of the solution (\ref{eq:7}) follows from Lemma \ref{le:2}. 
\end{proof}

\section{Conclusions}
In this section we will give simple consequences of the main theorem. First we will generalise the Dirichlet problem (\ref{eq:6})
to $B(a,r)$, where the boundary conditions are put on the set
$ a+\bigcup_{k=0}^{p-1} e^\frac{k\pi i}{p}\partial B(0,r) $. Applying Theorem \ref{th:1} to the function $v(x)=u(a+rx)$ we obtain
\begin{Wn}
\label{th:2}
The Dirichlet problem
\begin{equation}
\label{eq:13}
   \begin{cases}
   \Delta^pu(x)=0, & \text{} x\in B(a,r) \\
   u(x)=f(x), & \mbox{} x\in a+\bigcup_{k=0}^{p-1} e^\frac{k\pi i}{p}\partial B(0,r)
   \end{cases}
   \end{equation}
has the unique solution given by
\begin{equation}
\label{eq:general}
 u(x)=\frac{1}{p\omega_n}\sum_{k=0}^{p-1}\int\limits_{\partial B}
 \frac{ r^{2p}-|x-a|^{2p}}{r^{2p-n} \left| e^\frac{-k\pi i}{p}(x-a)-r\zeta \right|^n }f \left(a+re^\frac{k\pi i}{p}\zeta \right)dS(\zeta).
\end{equation}
\end{Wn}
\bigskip

As a simple consequence of Corollary \ref{th:2} we conclude that
\begin{Wn}[The mean value theorem for polyharmonic functions]  
Let $G$ be an open subset of $ \RR^n $. If $ u $ is a polyharmonic function on $G$ of order $p$ then
for every $a\in\RR^n$ and $r>0$ such that $\overline{B(a,r)}\subset G$ we get
\begin{equation}
\label{eq:mean}
u(a)=\frac{1}{p \omega_n}\sum_{k=0}^{p-1}\int\limits_{\partial B}u(a+e^\frac{k\pi i}{p}r\zeta)\,dS(\zeta).
\end{equation}
\end{Wn} 
\begin{proof}
First observe that there exist $\varepsilon>0$ such that $u$ is polyharmonic of order $p$ on $B(a,r+\varepsilon)$.
Applying Lemma \ref{le:1} to the function $v(\zeta):=u(a+(r+\varepsilon)\zeta)$ we conclude that
$u$ is analytic on $a+\bigcup_{k=0}^{p-1}e^{\frac{k\pi i}{p}}B(0,r)$ and continuous on
$a+\bigcup_{k=0}^{p-1}e^{\frac{k\pi i}{p}}\overline{B(0,r)}$. It means that $u$ is a solution of the Dirichlet problem (\ref{eq:13})
on $B(a,r)$ and we may use Corollary \ref{th:2}. So, by the formula (\ref{eq:general}) with $x=a$ and $f=u$ we obtain (\ref{eq:mean}).
\end{proof}

\begin{Uw}
 One can prove the same by using the Pizzetti type formula for polyharmonic functions (see \cite[Corollary 3 and Remark 6]{S-M}).
\end{Uw}

To show the next conclusion from the main theorem, we will consider the exterior Dirichlet problem, that is the problem where
the sought function is polyharmonic on the set $ (\RR^n\setminus B)$ and is also polyharmonic at infinity.
But first we recall two lemmas which we will use to solve the 
above mentioned problem. 
\begin{Le}[{\cite[Proposition 1.4]{A-C-L}}]
\label{le:3}
If $ \Delta^pu =0$ on $ B $, then the extended Kelvin transform of $u$ defined by
\begin{equation}
\label{eq:K}
K[u](x):=|x|^{2p-n}u\left( \frac{x}{|x|^2} \right)
\end{equation}
is a polyharmonic function of degree $ p $ on $(\RR^n\setminus B)\cup\{\infty\}$.
\end{Le}

The inverse of the extended Kelvin transform is of the same form (\ref{eq:K}). This leads us to the following definition of
polyharmonicity at infinity.
\begin{Df}[{\cite[Remark 1.4]{A-C-L}}]
 If $E\subset \RR^n$ is compact and $u$ is polyharmonic on $\RR^n\setminus E$, then $u$ is \emph{polyharmonic at $\infty$} provided
 $K[u]$ has a removable singularity at the origin.
\end{Df}

\begin{Le}[Symmetry Lemma]
\label{le:4}
For every $ x,y \in \CC^n$, $|x|\neq 0$ and $|y|\neq 0$,
the following formula 
\begin{equation*}
\left| \frac{x}{|x|}-|x|y \right|=\left| \frac{y}{|y|}-|y|x \right|
\end{equation*}
is valid.
\end{Le}
\begin{proof}
The proof is analogous as for $ \RR^n $ (cf. \cite[p.~10]{A-B-R}).
\end{proof}

\begin{Wn}
\label{th:3}
The exterior Dirichlet problem
 \begin{equation}
\label{eq:14}
   \begin{cases}
   \Delta^pu(x)=0, & \text{} x\in (\RR^n\setminus B)\cup\{\infty\} \\
   u(x)=f(x), & \mbox{} x\in \bigcup_{k=0}^{p-1}e^\frac{k\pi i}{p}\partial B
   \end{cases}
   \end{equation}
has the unique solution given by
\begin{equation*}
u\left( x \right) =-\frac{1}{p\omega_n}\sum_{k=0}^{p-1}\int\limits_{\partial B}
\frac{1-|x|^{2p}}{ \left| e^\frac{-k\pi i}{p}\zeta-x\right| ^n }f (e^\frac{k\pi i}{p}\zeta)d\,S(\zeta).
\end{equation*}
\end{Wn}

\begin{proof} 
By Theorem \ref{th:1} the solution of the interior Dirichlet problem corresponding to the exterior problem (\ref{eq:14}) has the form
$$v(x) =\frac{1}{p\omega_n}\sum_{k=0}^{p-1}\int\limits_{\partial B} \frac{1-|x|^{2p}}{|e^\frac{-k\pi i}{p}x-\zeta|^n }f (e^\frac{k\pi i}{p}\zeta) 
d\,S(\zeta)\qquad\textrm{for}\quad x\in B.$$
By Lemma \ref{le:3} the function $ K[v](x)$ is a solution of (\ref{eq:14}). So
\begin{eqnarray*}
u\left( x \right) & = & K[v](x)=\frac{1}{p\omega_n}\sum_{k=0}^{p-1}\int\limits_{\partial B}\frac{|x|^{2p}\left( 1-\left|\frac{x}{|x|^2}\right|^{2p}
\right) }{|x|^{n} \left|e^\frac{-k\pi i}{p}\frac{x}{|x|^2}-\zeta\right| ^n }f (e^\frac{k\pi i}{p}\zeta)  d\,S(\zeta) {}\\
& = & -\frac{1}{p\omega_n}\sum_{k=0}^{p-1}\int\limits_{\partial B} \frac{1-|x|^{2p}}{ \left|e^\frac{-k\pi i}{p}\right|_{\RR}^n
\left|\frac{x}{|x|}-e^\frac{k\pi i}{p}\zeta|x|\right|^n}f (e^\frac{k\pi i}{p}\zeta)\,dS(\zeta).
\end{eqnarray*}
By Lemma \ref{le:4} we get
\begin{eqnarray*}
u\left( x \right) & = & -\frac{1}{p\omega_n}\sum_{k=0}^{p-1}\int\limits_{\partial B}\frac{1-|x|^{2p}}{\left|e^\frac{-k\pi i}{p}
\right|_{\RR}^n \left| \frac{e^\frac{k\pi i}{p}\zeta}{|e^\frac{k\pi i}{p}\zeta|}-|e^\frac{k\pi i}{p}\zeta|x\right| ^n }f (e^\frac{k\pi i}{p}\zeta )\,  dS(\zeta){}\\
& = & -\frac{1}{p\omega_n}\sum_{k=0}^{p-1}\int\limits_{\partial B} \frac{1-|x|^{2p}}{\left|  e^\frac{-k\pi i}{p}\zeta- x\right| ^n }
f (e^\frac{k\pi i}{p}\zeta)\,dS(\zeta).
\end{eqnarray*}
The uniqueness is obvious.
\end{proof}

\section{Final remarks}
The existence and uniqueness of the solution of the Dirichlet problem (\ref{eq:1}) allows one to introduce the
so called spherical polyharmonics, which are  the most natural generalisation of spherical harmonics. Their theory will be studied in a subsequent paper.

The solution (\ref{eq:7}) of the Dirichlet problem (\ref{eq:1}) gives the connection between the Poisson formula (\cite[Theorem 1.17]{A-B-R})
for the harmonic functions
on the ball and the Cauchy-Hua formula ({\cite[Theorem 5.7]{M}}) for the holomorphic functions on the Lie ball. This connection will be also
investigated in a subsequent paper.

\section*{Acknowledgement}
The authors are grateful to Grzegorz {\L}ysik for suggesting the problem and for many valuable comments.

\end{document}